\documentclass{amsart}
\usepackage{amssymb,amsmath}

\usepackage[english]{babel}

\newtheorem{theorem}{Theorem}[section]
\newtheorem{lemma}[theorem]{Lemma}
\newtheorem{proposition}[theorem]{Proposition}
\newtheorem{corollary}[theorem]{Corollary}

\theoremstyle{definition}
\newtheorem{definition}[theorem]{Definition}
\newtheorem{example}[theorem]{Example}

\theoremstyle{remark}
\newtheorem{remark}[theorem]{Remark}

\numberwithin{equation}{section}

\addto\captionsenglish{}

\begin{document}

\title{Prime and primary ideals in characteristic one}

\author{{\bf Paul Lescot \rm}}

\address{LMRS \\
CNRS UMR 6085 \\
UFR des Sciences et Techniques \\
Universit\'e de Rouen \\
Avenue de l'Universit\'e BP12 \\
76801 Saint-Etienne du Rouvray (FRANCE)
Paul.Lescot@univ-rouen.fr \\
{http://www.univ-rouen.fr/LMRS/Persopage/Lescot/} \\
}

\date{01 September 2012}

\setcounter{section}{0}

\maketitle

\begin{abstract}

We study zero divisors and minimal prime ideals in semirings of characteristic one.
Thereafter we find a counterexample to the most obvious version of primary decomposition,
but are able to establish a weaker version.
Lastly, we study Evans'condition in this context.
\end{abstract}

\footnote{MSC : 06F05 , 20M12 , 08C99.}

\newpage

\section{Introduction}\label{S:intro}

Primary decomposition was first established in polynomial rings
(over $\mathbf Z$ or over a field) in Lasker's classical paper
(\cite{10}) ; another proof was later given by Macaulay (\cite{14}).
In her famous paper of 1921 (\cite{15}), Emmy Noether established the result for the class of rings that now bears her name.
Therefore Lasker's theorem led to the discovery of two of the main concepts of modern 
algebra : \it noetherian rings \rm and \it Cohen-Macaulay rings \rm .

The decomposition of an arbitrary ideal as an intersection of primary ones
is, \it via \rm the proof of Krull's Theorem, an essential tool in algebraic geometry (see \it e.g. \rm \cite{18}, pp.47--48).
The Riemann Hypothesis is arguably the most important open problem in mathematics ;
its natural analogue, Weil's conjecture (\cite{19}) was finally established by Deligne
(\cite{3}) using the whole strength of Grothendieck's \it theory of schemes\rm .

It has therefore long been expected (see \it e.g. \rm \cite{1} and \cite{17})
that an \lq\lq algebraic geometry in characteristic one\rq\rq might provide the 
natural framework for an approach of the Riemann Hypothesis. Many such theories have been propounded, including Deitmar's theory
of \it $F_{1}$--schemes \rm
(\cite{2}) and Zhu's \it characteristic one algebra \rm(\cite{20}). In \cite{11},\S 5, I have shown that part of Deitmar's theory embeds in a functorial
way into Zhu's ; the basic objects are $B_{1}$--algebras, 
\it i.e.  characteristic one semirings\rm , that is unitary semirings $A$ such that
$$
1_{A}+1_{A}=1_{A}\,\, .
$$
I have resolved to develop systematically and as far as possible the study of these objects.

As usual, we shall consider the set $B_{1}=\{0,1\}$ equipped with the usual multiplication
and addition, with the slight change that $1+1=1$.

In three previous articles (\cite{11}, \cite{12}, \cite{13}), we have shown that
$B_{1}$--algebras (\it i.e. \rm characteristic $1$ semirings) behave, in yet other
respects, like ordinary rings. In particular, one may define polynomial algebras over $B_{1}$
(\cite{11}, Theorem 4.5) and classify maximal congruences on them (\cite{11}, Theorem 4.8). 

There is a natural definition of a prime ideal (see \cite{12}, Definition 2.2)
in such a semiring ; the set $Pr_{s}(A)$of \it  saturated\rm (\cite{12}, Theorem 3.7)
prime ideals of $A$ can be endowed with a natural Zariski--type topology, to which most of the usual
topological properties of ring spectra carry over (see \cite{13}, Proposition 6.2).
In \cite{12}, we discussed the relationship between congruences and ideals in $B_{1}$--algebras ;
the two concepts are not equivalent, but \it excellent \rm congruences correspond bijectively
to \it saturated \rm ideals. The set $Pr_{s}(A)$ of saturated prime ideals of a $B_{1}$--algebra
$A$ is in bijection with the set $MaxSpec(A)$ of maximal (nontrivial) congruences on $A$;
that bijection is even a homeomorphism for the natural Zariski--type topologies (\cite{13}, Theorem 3.1),
and it is functorial (\cite{13}, Theorem 4.2).

It is therefore natural to examine whether higher results of commutative algebra
have analogs in the setting of $B_{1}$--algebras. Without any extra hypothesis, this is the case for the fundamental
properties of minimal (saturated) prime ideals (\S 3). Actually, modulo an hypothesis of
noetherian flavour, it appears that all minimal prime ideals (more generally, all associated prime
ideals) are saturated (\S 4).

The next natural question concerns a possible primary decomposition. The basic properties of
primary ideals carry over (\S 5), but Lasker--Noether primary decomposition need not hold, even though a weaker version can be established (\S 6). 
In other words, a (weakly) noetherian $B_{1}$--algebra
is not necessarily laskerian.

But it turns out (\S 7) that if the $B_{1}$--algebra is \it either \rm laskerian \it or \rm weakly
noetherian, it has the Evans property (first introduced in \cite{4}).

\section{Some definitions}
\label{sec:2}

We shall, except when explicitly mentioned otherwise, preserve the definitions and notation of
\cite{11}, \cite{12} and \cite{13}. We shall usually denote by $A$ a $B_{1}$--algebra ; 
$Pr(A)$ will denote the set of its prime ideals,
and $Pr_{s}(A)$ the set of its saturated prime ideals. $MinPr(A)$ and $MinPr_{s}(A)$
 will denote the sets of minimal elements (for inclusion) of $Pr(A)$ and $Pr_{s}(A)$,
 respectively. Classical arguments (see \it e.g. \rm \cite{7}, Proposition II.6, p.69)
establish that $(Pr(A),\supseteqq)$ and $(Pr_{s}(A),\supseteqq)$ are inductive.
 Therefore Zorn's Lemma implies that each prime (resp. prime saturated) ideal
 contains a minimal prime (resp. minimal prime saturated) ideal.

 By $Max_{s}(A)$ we denote the set of maximal elements among
 proper saturated ideals of $A$.
 
 The following two results are sometimes useful.
 \begin{proposition}
 $$
 Max_{s}(A)\subseteq Pr_{s}(A)\,\, .
 $$
 \end{proposition}
 \begin{proof}
 Let $\mathcal M\in Max_{s}(A)$, and let us assume $u\notin \mathcal M$,
 $v\notin \mathcal M$, and $uv\in \mathcal M$. Then the maximality of $\mathcal M$
 yields $\overline{\mathcal M+Au}=\overline{\mathcal M+Av}=A$. Therefore one may find $x\in \mathcal M+Au$ and $y\in \mathcal M+Av$
 such that $1+x=x$ and $1+y=y$. Let us write $x=m+au$ and $y=m^{'}+bv$($(m,m^{'})\in \mathcal M^{2}$,
 $(a,b)\in A^{2}$). Then
 $$
 uy=um^{'}+b(uv)\in \mathcal M
 $$
 and
 $$
 u+uy=u(1+y)=uy\in \mathcal M \,\, ,
 $$
 whence $u\in \overline{\mathcal M}=\mathcal M$, a contradiction :
 $\mathcal M$ is prime. 
 Arguments such as the above will often recur in this paper.
 
We might also have used Theorem 3.3 from \cite{13}.
 \end{proof}
 
 \begin{lemma}Let $I$ and $J$ denote ideals of $A$ ; then
 $$
r(\overline{I\cap J})=r(\overline{I}\cap \overline{J})=r(\overline{I})\cap r(\overline{J})\,\, .
$$
\end{lemma}
\begin{proof}
The inclusions 
$$
r(\overline{I\cap J})\subseteq r(\overline{I}\cap \overline{J})\subseteq r(\overline{I})\cap r(\overline{J})
$$
are clearly valid.
Let now $u\in r(\overline{I})\cap r(\overline{J})$ ; then $u^{m}\in\overline{I}$ for some $m\geq 1$
and $u^{n}\in \overline{J}$ for some $n\geq 1$. Thus one may find
$i\in I$ and $j\in J$ with $u^{m}+i=i$ and $u^{n}+j=j$.
Then $u^{m}j+ij=ij\in I \cap J$, whence $u^{m}j\in \overline{I\cap J}$.
But
$$
u^{m+n}+u^{m}j=u^{m}j
$$
whence
$$
u^{m+n}\in \overline{\overline{I\cap J}}=\overline{I\cap J}\,\, ,
$$
and 
$$
u\in r(\overline{I\cap J}) \,\, .
$$
Therefore
$$
r(\overline{I})\cap r(\overline{J})\subseteq r(\overline{I\cap J})\,\, ,
$$
and the result follows.
\end{proof}

 For $s\in A$ we define the \it annihilator \rm of $s$ by
 $$
 Ann_{A}(s):=\{x\in A \vert sx=0\}.
 $$
 It is clearly an ideal of $A$ ; furthermore, from $y\in  Ann_{A}(s)$ and $x+y=y$ follows
 $$
 sx=sx+0=sx+sy=s(x+y)=sy=0 \,\, ,
 $$
 thus $x\in Ann_{A}(s)$ : $Ann_{A}(s)$ is saturated.
 
  For $S$ a subset of $A$, we define
 $$
 Ann_{A}(S):=\bigcap_{s\in S}Ann_{A}(s) \,\, ;
 $$
 as an intersection of saturated ideals of $A$, it is a saturated ideal of $A$.
 
  For $x\in A\setminus\{0\}$, let
  $$
  \tilde{A}_{x}:=\displaystyle\frac{A}{\mathcal R_{Ann_{A}(x)}}\,\, ,
  $$
  and let $\pi_{x}:A\twoheadrightarrow \tilde{A}_{x}$ denote the canonical
  projection.
  \begin{definition}An ideal $\mathcal P$ of $A$ is termed \it associated \rm to $x\in A\setminus \{0\}$ if it can
  be expressed as $\mathcal P=\pi_{x}^{-1}(\mathcal Q)$ for some 
  minimal prime ideal $\mathcal Q$ of $\tilde{A}_{x}$ ; it is termed \it associated \rm if it is associated to some
  $x\in A\setminus \{0\}$.
  \end{definition}
  $Ass(A)$ will denote the set of associated ideals of $A$ ; clearly, 
  $$
  Ass(A)\subseteq Pr(A)\,\, .
  $$
  
  Obviously, each minimal prime ideal of $A$ is associated ($x=1$ is suitable),
  whence $MinPr(A)\subseteq Ass(A)$.
 
 The elements of the set
 $$
 \mathcal D_{A}:=\{a\in A\setminus \{0\} \,\, \vert \,\, (\exists b\in A\setminus \{0\}) \,\,  ab=0 \}
 $$
 are called \it zero--divisors \rm in $A$.
 Clearly, for $A$ non--trivial, one has
 $$
 \mathcal D_{A}\cup\{0\}=\bigcup_{s\in A\setminus\{0\}}Ann_{A}(s)\,\, .
 $$
 An ideal $I$ of $A$ will be termed \it radical \rm if $I=r(I)$ ;
 from \cite{13}, Proposition 5.5, it follows easily that a saturated 
 ideal is radical if and only if it is an intersection of prime (saturated)
 ideals.
\begin{definition}The $B_{1}$--algebra $A$ is \it noetherian \rm 
if every ascending chain of ideals of $A$ is ultimately stationary.
\end{definition}

By standard arguments (see \it e.g. \rm \cite{8}, Proposition I.2, p.47), 
this is equivalent to the assertion
that every ideal of $A$ is finitely generated.

\begin{definition}The $B_{1}$--algebra $A$ is \it weakly noetherian \rm 
if every ascending chain of saturated ideals of $A$ is ultimately stationary.
\end{definition}

It is obvious that every noetherian $B_{1}$--algebra is weakly noetherian.
The converse is false ; in fact, $B_{1}[x]$ is weakly noetherian
(it follows from the reasoning used in the proof of \cite{11}, Theorem 4.2 that its saturated ideals are 
$\{0\}$ and the $x^{n}B_{1}[x](n\in \mathbf N$)) but not noetherian
(one may even find, in $B_{1}[x]$, a strictly increasing sequence of \bf prime \rm
ideals : cf. \cite{9}, ch.3, p.65).

\begin{definition}We call the $B_{1}$--algebra $A$ \it standard \rm if $\mathcal D_{A}\cup\{0\}$
is a finite union of saturated prime ideals of $A$.
\end{definition}

\begin{definition}For $J$ an ideal of $A$ and $x\in A$, let
$$
\mathcal C_{x}(J):=\{y\in A \vert xy\in J \}.
$$
\end{definition}
Clearly, $\mathcal C_{x}(J)$ is an ideal of $A$, saturated whenever $J$ is ; furthermore, 
$$\mathcal C_{x}(\{0\})=Ann_{A}(x)\,\, .$$
For saturated $J$, the ideals of the form $\mathcal C_{x}(J)$($x\notin J$) will be termed \it $J$--conductors \rm .

\begin{lemma}Let $J$ be a saturated ideal of $A$, $y\notin J$, and assume that
$$
\mathcal P:=\mathcal C_{y}(J)
$$ 
is a maximal element of the set of
$J$--conductors. Then $\mathcal P\in Pr(A)$.
\end{lemma}
\begin{proof}

One has $1\notin \mathcal P$ (as $y\notin J$), whence
$\mathcal P\neq A$. Let us assume $uv\in \mathcal P$ and 

$u\notin \mathcal P$ ;
then $uy\notin J$ and $\mathcal C_{y}(J)\subseteq \mathcal C_{uy}(J)$.
It follows that
$$
\mathcal P=\mathcal C_{y}(J)=\mathcal C_{uy}(J)\,\, ;
$$
but
$$
v(uy)=(uv)y\in J \,\, ,
$$
whence
$$
v\in C_{uy}(J)=\mathcal P \,\, :
$$
$\mathcal P$ is prime.
\end{proof}

\section{Minimal prime ideals}
\label{sec:3}
In this paragraph, $A$ will denote an arbitrary $B_{1}$--algebra.

\begin{theorem} Let $\mathcal P\in MinPr(A)\cup MinPr_{s}(A)$ ; then each nonzero element of $\mathcal P$
is a zero--divisor in $A$.
\end{theorem}
\begin{proof}Let $x\in \mathcal P\in MinPr(A)$, $x\neq 0$, and assume that $x$ is not a zero--divisor ; then, for
each $a\in A\setminus \{0\}$ and $n\in \mathbf N$, $ax^{n}\neq 0$.
In particular
$$
\forall n\in \mathbf N \,\, \forall a\in A\setminus \mathcal P  \,\, ax^{n}\neq 0.
$$
Let $\mathcal E$ denote the set of ideals $I$ of $A$ such that
$$
\forall n\in \mathbf N \,\, \forall a\in A\setminus \mathcal P \,\,  ax^{n}\notin I \, . \,\,\,\, (*)
$$
Then $\mathcal E\neq \emptyset$ (as $\{0\}\in \mathcal E$), and $\mathcal E$ is inductive for $\subseteq$, whence $\mathcal E$
contains a maximal element $I$. As $1=1.x^{0}\notin I$, $I\neq A$.

Let us suppose for a moment that $uv\in I$, $u\notin I$ and $v\notin I$ ; then $I+Au$ and $I+Av$
are ideals of $A$ strictly containing $I$, whence $I+Au\notin \mathcal E$ and $I+Av\notin \mathcal E$.
Thus one may find $(a,b)\in (A\setminus \mathcal P)^{2}$, $(i,j)\in I^{2}$, $(c,d)\in A^{2}$ and $(m,n)\in \mathbf N^{2}$
with $ax^{m}=i+cu$ and $bx^{n}=j+dv$. 

Then $ab\in A\setminus \mathcal P$ (as $\mathcal P$ is prime)
and
\begin{eqnarray}
abx^{m+n}
&=&(ax^{m})(bx^{n}) \nonumber \\
&=&(i+cu)(j+dv) \nonumber \\
&=&i(j+dv)+(cu)j+(cd)(uv)\in I \,\, , \nonumber
\end{eqnarray}
a contradiction. Therefore $I$ is prime.
But, by definition,
$$
\forall a\in A\setminus \mathcal P  \,\, a=ax^{0}\notin I,
$$
whence $A\setminus \mathcal P\subseteq A\setminus I$, or $I\subseteq \mathcal P$.
The minimality of $\mathcal P$ now implies that
$$
I=\mathcal P \,\, ,
$$
whence $1.x^{1}=x\in \mathcal P=I$,
contradicting the definition of $I$
(we have essentially followed \cite{5}, Corollary 1.2,
and \cite{8}, p. 34, Lemma 3.1).

In case $\mathcal P\in MinPr_{s}(A)$, the same argument applies modulo a slight complication : 
by defining $\mathcal E$ to be the set of \it saturated \rm ideals $I$ of $A$ satisfying $(*)$,
we find a maximal element $I$ of $\mathcal E$, and have $I\neq A$. Assuming $uv\in I$, $u\notin I$
and $v\notin I$, we see that that $\overline{I+Au}\notin \mathcal E$ and $\overline{I+Av}\notin \mathcal E$.
Therefore we may find $(a,a^{'})\in (A\setminus \mathcal P)^{2}$ and $(m,n)\in \mathbf N^{2}$
such that $ax^{m}\in \overline{I+Au}$ and $a^{'}x^{n}\in \overline{I+Av}$.
Therefore $ax^{m}+y=y$ for some $y\in I+Au$,
and $a^{'}x^{n}+z=z$ for some $z\in I+Av$. Set $y=i+cu$ ($i\in I$) and $z=i^{'}+dv$ ($i^{'}\in I$) ;
then
$$
yv=iv+c(uv)\in I \,\, .
$$
As
\begin{eqnarray}
ax^{m}v+yv
&=&(ax^{m}+y)v \nonumber \\
&=&yv \nonumber
\end{eqnarray}
and $I$ is saturated, it appears that $ax^{m}v\in I$.
Then

\begin{eqnarray}
ax^{m}z
&=&ax^{m}(i^{'}+dv) \nonumber \\
&=&ax^{m}i^{'}+d(ax^{m}v)\nonumber \\
&\in& I \,\, ; \nonumber
\end{eqnarray}
as
\begin{eqnarray}
aa^{'}x^{m+n}+ax^{m}z
&=&ax^{m}(a^{'}x^{n}+z) \nonumber \\
&=&ax^{m}z \,\, ,\nonumber 
\end{eqnarray}
it follows as above that $aa^{'}x^{m+n}\in I$. But $aa^{'}\in A \setminus \mathcal P$, contradicting the
definition of $I$.

\end{proof}

\section{The weakly noetherian case}
\label{sec:4}

\begin{theorem}In case $A$ is weakly noetherian, each associated prime ideal of $A$ is of the
form $Ann_{A}(u)$ for some $u\in A\setminus \{0\}$ ; in particular, it is saturated.
\end{theorem}
\begin{proof}
Let $\mathcal P$ denote a prime ideal of $A$ associated to $x\in A\setminus\{0\}$ ; then 
$$
\mathcal P=\pi_{x}^{-1}(\mathcal Q)
$$ 
for some $Q\in MinPr(\tilde{A}_{x})$.
We define
$$
\mathcal W(\mathcal P):=\{z\in A \vert Ann_{A}(zx)\subseteq \mathcal P\}\,\, .
$$

$\mathcal W(\mathcal P)$ is non--empty, as $1\in \mathcal W(\mathcal P)$.
For $y\in \mathcal W(\mathcal P)$, let
$$
I_{\mathcal P}(y):=\bigcup_{s\in A\setminus \mathcal P}Ann_{A}(sxy) \,\, .
$$
As
$$
\forall (s,s^{'})\in (A\setminus \mathcal P)^{2}
$$
$$
Ann_{A}(sxy)\cup Ann_{A}(s^{'}xy)\subseteq Ann_{A}(ss^{'}xy) \,\, 
$$
and $ss^{'}\in A \setminus \mathcal P$, $I_{\mathcal P}(y)$ is the union of a filtering family
of saturated ideals, whence it is itself a saturated ideal.  

By definition, whenever $y\in \mathcal W(\mathcal P)$, $Ann_{A}(xy)\subseteq \mathcal P$, therefore
from \newline
$s\in A\setminus \mathcal P$ and $z\in Ann_{A}(sxy)$ follows
$$
(sz)(xy)=(sxy)z=0 \,\, ,
$$
thus $sz\in Ann_{A}(xy)\subseteq \mathcal P$, $sz\in\mathcal P$ and $z\in \mathcal P$. We have shown that
$$
I_{\mathcal P}(y)\subseteq \mathcal P \,\, .
$$

Let now $J:=I_{\mathcal P}(y_{0})$ denote a maximal element of
$$
\{I_{\mathcal P}(y)\vert y\in \mathcal W(\mathcal P)\}\,\, 
$$
(the existence of such an element follows from the weak noetherianity hypothesis).
As seen above, $J\subseteq \mathcal P$, whence $J\neq A$.
Let us suppose $ab\in J$ and $a\notin J$ ; then, for each $s\in A\setminus \mathcal P$,
$a\notin Ann_{A}(sxy_{0})$, whence
$$
s(xy_{0}a)=(sxy_{0})a \neq 0 \,\, .
$$
Therefore $Ann_{A}(xy_{0}a)\subseteq \mathcal P$, \it i.e. \rm $y_{0}a\in \mathcal W(\mathcal P)$.
Clearly $I_{\mathcal P}(y_{0})\subseteq I_{\mathcal P}(y_{0}a)$, whence $I_{\mathcal P}(y_{0})= I_{\mathcal P}(y_{0}a)$
according to the definition of $y_{0}$.

As $ab\in J=I_{\mathcal P}(y_{0})$, there exists $s\in A\setminus \mathcal P$ such that $(sxy_{0})ab=0$ ;
but then $s(xy_{0}a)b=0$, whence $b\in Ann_{A}(sx(y_{0}a))\subseteq I_{\mathcal P}(y_{0}a)=I_{\mathcal P}(y_{0})=J$.
We have shown that $ab\in J$ implies $a\in J$ or $b\in J$ : $J$ is prime.

As $J\subseteq \mathcal P$ and
$$
Ann_{A}(x)\subseteq Ann_{A}(xy_{0})\subseteq I_{\mathcal P}(y_{0})=J \,\, ,
$$
$\pi_{x}(J)$ is a prime ideal of $\tilde{A}_{x}$ and
$\pi_{x}(J)\subseteq \pi_{x}(\mathcal P)=\mathcal Q$, it now follows from the minimality of $\mathcal Q$ that 
$\pi_{x}(J)=\mathcal Q$.

Let now $u\in \mathcal P$ ; then $\pi_{x}(u)\in \pi_{x}(\mathcal P)=\mathcal Q=\pi_{x}(J)$,
therefore $\pi_{x}(u)=\pi_{x}(j)$ for some $j\in J$. Then there is $y\in Ann_{A}(x)$ such that $u+y=j+y$,
hence $u+y\in J$, and (as $u+(u+y)=u+y$), $u\in \overline{J}=J$.
It follows that $\mathcal P \subseteq J$, whence 
$\mathcal P=J=I_{\mathcal P}(y_{0})$;
in particular, $\mathcal P$ is saturated.

As $A$ is weakly noetherian, there is a finite family $(p_{1},...,p_{n})$ of elements
of $\mathcal P$ such that
$$
\mathcal P=\overline{<p_{1},...,p_{n}>} \,\, .
$$

Each $p_{j}$ belongs to $\mathcal P=J=I_{\mathcal P}(y_{0})$, whence there is
an $s_{j}\in A \setminus \mathcal P$ such that $p_{j}\in Ann_{A}(s_{j}xy_{0})$.
Let $s_{0}:=s_{1}...s_{n}$ and
$$
u:=s_{0}xy_{0}=s_{1}...s_{n}xy_{0} \,\, ;
$$
then each $p_{j}$ belongs to $Ann_{A}(u)$, whence
\begin{eqnarray}
\mathcal P
&=&\overline{<p_{1},...,p_{n}>}\nonumber \\
&\subseteq& \overline{Ann_{A}(u)} \nonumber \\
&=& Ann_{A}(u) \nonumber
\end{eqnarray}
(as $Ann_{A}(u)$ is saturated).

On the other hand, $s_{0}\in A \setminus \mathcal P$, therefore

\begin{eqnarray}
Ann_{A}(u)
&=&Ann_{A}(s_{0}xy_{0})\nonumber \\
&\subseteq&  I_{\mathcal P}(y_{0}) \nonumber \\
&=& \mathcal P \,\, , \nonumber
\end{eqnarray}
whence $\mathcal P=Ann_{A}(u)$.
\end{proof}
\begin{corollary}If $A$ is weakly notherian, then $MinPr(A)=MinPr_{s}(A)$\,\,\, .
\end{corollary}
\begin{proof}
Let $\mathcal P\in MinPr(A)$. As seen in \S 2, $\mathcal P$ is associated,
whence, by Theorem 4.1, $\mathcal P$ is saturated, hence $\mathcal P\in MinPr_{s}(A)$.

Conversely, let $\mathcal P\in MinPr_{s}(A)$ ; then $\mathcal P$ is prime, hence
contains some minimal prime ideal $\mathcal P_{0}$ (cf. \S 2). Now $\mathcal P_{0}$ is saturated
whence (as $\mathcal P_{0}\subseteq \mathcal P$) $$\mathcal P = \mathcal P_{0}\in MinPr(A)\,\, .$$
\end{proof}

\section{Definition and first properties of primary ideals}
\label{sec:5}
The usual theory generalizes without major problem to $B_{1}$--algebras.

\begin{definition}An ideal $\mathcal Q$ of $A$ is termed \it primary \rm
if $\mathcal Q\neq A$ and
$$
\forall (x,y)\in A^{2} \,\, [xy\in \mathcal Q \implies x\in \mathcal Q \,\, \text{or} \,\, (\exists n\geq 1)\,\, y^{n}\in \mathcal Q].
$$
\end{definition}

Obviously, a prime ideal is primary.

\begin{proposition}If $\mathcal Q$ is primary, then $r(\mathcal Q)$ is prime.
\end{proposition}

\begin{proof}Let $\mathcal P=r(\mathcal Q)$. As $\mathcal Q\neq A$, $1\notin\mathcal Q$, thus
$1\notin \mathcal P$. Let us assume that $uv\in \mathcal P$ ; then, for some $n\geq 1$, $(uv)^{n}\in \mathcal Q$,
\it i.e. \rm $u^{n}v^{n}\in \mathcal Q$, whence (as $\mathcal Q$ is primary) either $u^{n}\in \mathcal Q$
or there exists $m\geq 1$ with $v^{nm}=(v^{n})^{m}\in \mathcal Q$. Therefore either $u$ or $v$
belongs to $r(\mathcal Q)=\mathcal P$ : $\mathcal P$ is prime.
\end{proof}
\begin{remark}As seen in \cite{12}, Lemma 4.4(ii), if $\mathcal Q$ is saturated
then so is $\mathcal P=r(\mathcal Q)$.
\end{remark}
\begin{definition}The primary ideal $\mathcal Q$ will be termed \it $\mathcal P$--primary \rm
if $\mathcal P=r(\mathcal Q)$.
\end{definition}
\begin{lemma}Let $\mathcal Q_{1}$,...,$\mathcal Q_{n}$ be $\mathcal P$--primary ideals for the same prime 
ideal $\mathcal P$ ; then $\mathcal Q:=\mathcal Q_{1}\cap ... \cap \mathcal Q_{n}$ is also $\mathcal P$--primary.
\end{lemma}
\begin{proof}Let us assume $xy\in \mathcal Q$ and $x\notin \mathcal Q$.
As $x\notin \mathcal Q$, there is a $k\in\{1,...,n\}$ such that $x\notin \mathcal Q_{k}$ ;
as $xy\in \mathcal Q\subseteq \mathcal Q_{k}$, we have $xy\in \mathcal Q_{k}$, whence
(as $\mathcal Q_{k}$ is primary) there exists $n\geq 1$ such that $y^{n}\in \mathcal Q_{k}$.
Then $y\in r(\mathcal Q_{k})=\mathcal P$. As all $\mathcal Q_{i}$'s are $\mathcal P$--primary,
one has, for each $i$, $y\in r(\mathcal Q_{i})$, whence there exists $m_{i}\geq 1$ such that
$y^{m_{i}}\in \mathcal Q_{i}$. Let $m^{'}:=\max_{1\leq i \leq n}(m_{i})$ ;
then $$y^{m^{'}}\in \mathcal Q_{1}\cap ... \cap \mathcal Q_{n}=\mathcal Q\,\, : $$
$\mathcal Q$ is primary.

Incidentally, we have established that $\mathcal P \subseteq r(\mathcal Q)$ ; but
$r(\mathcal Q)\subseteq r(\mathcal Q_{1})=\mathcal P$, whence $\mathcal P=r(\mathcal Q)$ : $\mathcal Q$
is $\mathcal P$--primary.
 
\end{proof}

\section{Weak primary decomposition}
\label{sec:6}
\begin{definition}The $B_{1}$--algebra $A$ is termed \it laskerian \rm if any saturated ideal
of $A$ can be expressed as a finite intersection of saturated primary ideals.
\end{definition}

It is natural to conjecture that each weakly noetherian $B_{1}$--algebra is laskerian,
but this is false, as shown by the following example.

\begin{example}Let $A=\{0,z,x,y,u,1\}$ ; it is easily seen that there 
is a unique structure of $B_{1}$--algebra on $A$ such that 
$z+x=x$, $z+y=y$, $x+y=u$, $u+1=1$,
$x^{2}=x$, $y^{2}=y$, $z^{2}=0$, $u^{2}=u$, $xy=xz=yz=uz=0$, $xu=x$
and $yu=y$. 

Each saturated primary ideal of $A$ contains $0=xy$, therefore it contains either
$x$ or a power of $y$, whence it contains $x$ or $y$, thus it contains $z$. Therefore 
any intersection of saturated primary ideals contains $z$, and $\{0\}$ is not an intersection
of saturated primary ideals : $A$ is not laskerian.
\end{example}

Nevertheless a weaker property holds true.
\begin{theorem}Let $I$ denote a saturated radical ideal in the $B_{1}$--algebra $A$ ;
then $I$ can be written as a finite intersection of prime saturated ideals.
\end{theorem}
\begin{proof}Let us proceed by contradiction, and let $J$ denote a maximal element
of the set of saturated prime ideals that cannot be written as a finite intersection of prime saturated ideals ;
in particular, $J\neq A$ and $J$ is not prime. Therefore one may find $u\notin J$ and $v\notin J$
with $uv\in J$. Let $K=\overline{J+Au}$ and $L=\overline{J+Av}$ ; then $K$ and $L$ are saturated ideals of
$A$ strictly containing $J$, whence each is a finite intersection of prime saturated ideals.
Clearly, $J\subseteq K\cap L$. Let now $x\in K\cap L$ ; then $x+y=y$ for some $y\in J+Au$
and $x+z=z$ for some $z\in J+Av$ ; writing $y=j+au$ ($j\in J$) and $z=j^{'}+bv$ ($j^{'}\in J$)
we get 
$$
vx+vy=v(x+y)=vy \,\, ; 
$$
but $vy=v(j+au)=vj+a(uv)\in J$, whence $vx\in\overline{J}=J$.
But then 
$$
xz=x(j^{'}+bv)=xj^{'}+b(xv)\in J\,\, ,
$$
and 
$$
x^{2}+xz=x(x+z)=xz\,\, .
$$
It follows that $x^{2}\in\overline{J}=J$, whence $x\in r(J)=J$.
We have shown that $J=K\cap L$, hence $J$ is a finite intersection of
saturated prime ideals, a contradiction.

\end{proof}
\begin{corollary}If $A$ is weakly noetherian and $I$ is a saturated ideal of $A$,
there are saturated prime ideals $\mathcal P_{1}$,...,$\mathcal P_{n}$ of $A$
such that
$$
r(I)=\mathcal P_{1}\cap ... \cap \mathcal P_{n}\,\, .
$$
\end{corollary}
\begin{proof}According to \cite{12}, Lemma 4.4.(ii), $r(I)$ is saturated
(and, of course, radical), and we may apply Theorem 6.3.
\end{proof}
\begin{proposition}If $A$ is weakly noetherian, then $MinPr(A)$ is finite.
\end{proposition}
\begin{proof}Let us apply Corollary 6.3 to $I=\{0\}$ ; we obtain the existence of a finite family
$\mathcal P_{1}$,...,$\mathcal P_{n}$ of saturated prime ideals of $A$ such that
$$
Nil(A)=r(\{0\})=\mathcal P_{1}\cap ...\cap \mathcal P_{n}\,\, .
$$

Let us suppose that no $\mathcal P_{j}$($1\leq j \leq n$) be contained in $\mathcal P$;
then one may find, for each $j\in\{1,...,n\}$, $x_{j}\in \mathcal P_{j}$, $x_{j}\notin \mathcal P$.
It ensues that
$$
x_{1}...x_{n}\in \mathcal P_{1}\cap ...\cap \mathcal P_{n}=Nil(A)\subseteq \mathcal P \,\,
$$
and each $x_{j}\notin \mathcal P$, contradicting the definition of $\mathcal P$.
Therefore, for some $j$, $\mathcal P_{j}\subseteq \mathcal P$, whence (by the minimality of $\mathcal P$)
$\mathcal P_{j}=\mathcal P$. We have shown that
$$
MinPr(A)\subseteq \{\mathcal P_{1},...,\mathcal P_{n}\} \,\, ;
$$
in particular, $MinPr(A)$ is finite.
\end{proof}
\begin{remark}Incidentally, we have reestablished Corollary 4.2, as all 
$\mathcal P_{j}$'s are saturated.
\end{remark}

\section{The Evans condition}
\label{sec:7}

For $A$ a $B_{1}$--algebra and $I$ an ideal of $A$, let
\begin{eqnarray}
\mathcal D_{A}(I) \nonumber
&:=&\{x\in A \vert (\exists y\notin I)\,\, xy\in I\} \nonumber \\
&=&\bigcup_{y\notin I}\mathcal C_{y}(I) \,\, .\nonumber
\end{eqnarray}
Obviously, if $A$ is nontrivial, 
$$
\mathcal D_{A}(\{0\})=\mathcal D_{A}\cup\{0\}\,\, .
$$
\begin{definition}(see \cite{4}, and also \cite{8}, ch.3, pp.121--122) 

The $B_{1}$--algebra $A$ \it has the Evans property \rm
if, for each saturated prime ideal $I$ of $A$, $\mathcal D_{A}(I)$ is a finite union of
saturated prime ideals of $A$.
\end{definition}
\begin{remark}(to be compared with Theorem 3.1) If $A$ has the Evans property, then $A$
is standard (take $I=\{0\}$).
\end{remark}
\begin{theorem}If $A$ is laskerian, then it has the Evans property.
\end{theorem}
\begin{proof}We follow closely \cite{4}, Proposition 7.
Let $I$ denote a saturated ideal of $A$ ; then one may write
$$
I=\mathcal Q_{1}\cap ... \cap \mathcal Q_{n} \,\, (n\in \mathbf N)
$$
where each $\mathcal Q_{i}$ is saturated and primary.
Let us choose such a decomposition with $n$ minimal, and, for each $j$, set
$\mathcal P_{j}=r(\mathcal Q_{j})$ ; according to Proposition 5.2,
$\mathcal P_{j}$ is prime (and saturated).

Let $y\in \mathcal D_{A}(I)$ ; there is $x\notin I$ such that $yx\in I$. As $x\notin I=\mathcal Q_{1}\cap ... \cap \mathcal Q_{n}$, there exists
$j\in\{1,...,n\}$ such that $x\notin \mathcal Q_{j}$. 

As $xy=yx\in I\subseteq \mathcal Q_{j}$, 
$xy\in \mathcal Q_{j}$ ; therefore, as $\mathcal Q_{j}$ is primary, there is a $m\geq 1$ such that $y^{m}\in \mathcal Q_{j}$, whence
$y\in r(\mathcal Q_{j})=\mathcal P_{j}$. We have shown that
$$
\mathcal D_{A}(I)\subseteq \mathcal P_{1}\cup ... \cup \mathcal P_{n} \,\, .
$$

Conversely, let $y\in \mathcal P_{1}\cup ... \cup \mathcal P_{n}$; then $y\in \mathcal P_{j}$ for some
$j$. Let
$$
K_{j}=\bigcap_{i=1 ;i\neq j}^{n}\mathcal Q_{i}\,\, ;
$$
according to our choice of $n$, $K_{j}\neq I$; as $I=K_{j}\cap \mathcal Q_{j}$, one has
$K_{j}\nsubseteq \mathcal Q_{j}$, whence there exists $b\in K_{j}$, $b\notin \mathcal Q_{j}$. As 
$y\in \mathcal P_{j}$, $y^{m}\in \mathcal Q_{j}$ for some $m\geq 1$,
therefore $y^{m}b\in K_{j}\cap \mathcal Q_{j}=I$.

But $y^{0}b=b\notin I$ (as $b\notin \mathcal Q_{j}$) ; therefore there is a (unique) $k\in\mathbf N$
such that $y^{k}b\notin I$ and $y^{k+1}b\in I$.
Let $z:=y^{k}b$; then $z\notin I$ and $yz=y^{k+1}b\in I$, hence
$y\in \mathcal D_{A}(I)$.
Thus
$$
\mathcal D_{A}(I)=\mathcal P_{1}\cup ... \cup \mathcal P_{n} \,\, ,
$$
as desired.
\end{proof}
\begin{theorem}
If $A$ is weakly noetherian, then it has the Evans property.
\end{theorem}
\begin{proof}

We shall adapt the reasoning used in the proof of Lemma 7
from \cite{16}.

Let $I$ denote a saturated ideal of $A$ ; according to the weak noetherianity hypothesis, each $I$--conductor
is contained in a maximal one.

Let $\mathcal E$ denote the set of $y\in A\setminus I$ such that
$\mathcal C_{y}(I)$ is maximal, and let
$$
R=\overline{<\mathcal E>}\,\, .
$$

Using once more the weak noetherianity hypothesis, one finds
a finite family $(y_{1},...,y_{n})\in \mathcal E^{n}$
such that
$$
R=\overline{<y_{1},...,y_{n}>}\,\, .
$$
By definition of $\mathcal E$ and Lemma 2.8, each $P_{j}:=\mathcal C_{y_{j}}(I)$ is prime.
Let 
$$
u\in \mathcal P_{1}\cap ... \cap \mathcal P_{n}\,\, ;
$$
then, by definition, $uy_{j}\in I$ for each $j$, whence
$$
<y_{1},...,y_{n}>\subseteq \mathcal C_{u}(I)
$$
and
$$
R=\overline{<y_{1},...,y_{n}>}\subseteq \mathcal C_{u}(I)
$$
(as $\mathcal C_{u}(I)$ is saturated).

Let now $x\in\mathcal E$ and $\mathcal P=\mathcal C_{x}(I)$ ; then $x\in R$, whence
$x\in \mathcal C_{u}(I)$ and $ux\in I$. It follows that $u\in \mathcal C_{x}(I)=\mathcal P$.
Therefore
$$
\mathcal P_{1}\cap ... \cap \mathcal P_{n}\subseteq \mathcal P, 
$$
thus
$$
\mathcal P_{1}... \mathcal P_{n}\subseteq \mathcal P \,\, ;
$$
as in the proof of Proposition 6.5, it follows that, for some $j$,
$\mathcal P_{j}\subseteq \mathcal P$, whence (by maximality) $\mathcal P_{j}=\mathcal P$. Therefore
the set of maximal $I$--conductors is contained in $\{\mathcal P_{1},...,\mathcal P_{n}\}$ ;
in particular, it is finite.

Thus, the maximal elements of $\mathcal E$ are finite in number ; but they are prime ideals
(Lemma 2.8), and $\mathcal D_{A}(I)$ is their union.

\end{proof}


\section{Bibliography}



\begin{thebibliography}{19}

\bibitem{1} A. Connes and C. Consani \textit{Schemes over $F_{1}$ and zeta functions}, Compositio Mathematica 146 (6)(2010), pp.1383--1415. 

\bibitem{2} A.Deitmar \textit{Schemes over $F_{1}$}, in {\it Number Fields and Function Fields - two parallel worlds}, pp. 87-100, Birkha\"user, Boston, 2005.

\bibitem{3} P. Deligne \textit{La conjecture de Weil. I},Publ. Math. IHES 43 (1974), pp. 273-307.

\bibitem{4} E. G. Evans \textit{Zero--divisors in noetherian--like rings},
Trans. Amer. Math. Soc. 155(2), 1971, pp.505--512.

\bibitem{5} M. Henriksen, M. Jerison  \textit{The space of minimal prime ideals of a commutative
ring}, Trans. Amer. Math. Soc. 115, 1969, pp.110--130.

\bibitem{6} J. Kist \textit{Minimal prime ideals in commutative semigroups},
Proc. London Math. Soc.(3)13, 1963, pp. 31--50.

\bibitem{7} J.-P. Lafon \textit{Les formalismes fondamentaux de l'alg\`ebre commutative},
Hermann, 1998.

\bibitem{8} J.-P. Lafon \textit{Alg\`ebre commutative. Langages g\'eom\'etrique et alg\'ebrique},
Hermann, 1977.

\bibitem{9} S. LaGrassa \textit{Semirings : ideals and polynomials}, PhD thesis, University of Iowa, 1995.

\bibitem{10} E. Lasker \textit{Zur Theorie der Moduln und Ideale}, Mathematische Annalen 60, 1906, pp. 20--116.

\bibitem{11} P. Lescot \textit{Alg\`ebre Absolue}, 
Ann. Sci. Math. Qu\'ebec 33(2009), no 1, pp. 63-82.
   
\bibitem{12} P. Lescot \textit{Absolute Algebra II-Ideals and Spectra},
Journal of Pure and Applied Algebra 215(7), 2011, pp. 1782--1790.

\bibitem{13} P. Lescot \textit{Absolute Algebra III-The saturated spectrum},
Journal of Pure and Applied Algebra 216, 2012, pp. 1004--1015.

\bibitem{14} F.S. Macaulay \textit{Algebraic Theory of Modular Systems}, Cambridge Tracts no 19, Cambridge, 1916.

\bibitem{15} E. Noether \textit{Idealtheorie in Ringbereichen}, Mathematische Annalen 83, 1921, pp. 24-66.

\bibitem{16} B.L.Osofsky \textit{Noether Lasker Primary Decomposition Revisited}, Amer. Math. Monthly
101(8), 1994, pp. 759--768.

\bibitem {17} C. Soul\'e \textit{Les vari\'et\'es sur le corps \`a un \'el\'ement}, Moscow Math. Journal,  Vol. 4, no 1, 2004, pages 217-244.

\bibitem{18} A. Weil \textit{Foundations of algebraic geometry (2nd edition)}, Am. Math. Soc. Coll., vol. XXIX, New York, 1962.

\bibitem{19} A. Weil \textit{Numbers of solutions of equations in finite fields}, Bull. Amer. Math. Soc. 55, 1949, pp. 497--508.

\bibitem{20}
  Y. Zhu
    \textit{Combinatorics and characteristic one algebra}, preprint, 2000. 
\end{thebibliography}
\end{document}